\documentclass{article}
\usepackage[utf8]{inputenc}

\usepackage{amssymb,amsmath,amsthm,mathabx}
\usepackage{wasysym}
\usepackage{graphicx}
\usepackage{proof}
\usepackage{cite}
\usepackage[UKenglish]{babel}
\usepackage{color}
\usepackage[all]{xy} 
\usepackage{hyperref}
\usepackage{times}

\usepackage{thmtools, thm-restate}
\hypersetup{
   colorlinks,%
   citecolor=blue,%
   filecolor=black,%
   linkcolor=red,%
   urlcolor=black
 }

 \newcommand{\lr}[1]{\langle #1 \rangle}
\newcommand{\lra}{\leftrightarrow}
\newcommand{\Kw}{\textit{Kw}}

 \newcommand{\K}{\ensuremath{\textit{K}}}
 \newcommand{\hK}{\ensuremath\hat{{\textit{K}}}}
\newcommand{\I}{\ensuremath{\textit{I}}}

\newcommand{\M}{\ensuremath{\mathcal{M}}}

\newcommand{\BP}{\textbf{P}}
\newcommand{\Ag}{\textbf{AG}}
\newcommand{\IRI}{\mathbf{IRI}}
\newcommand{\LI}{\mathbf{LI}}

\renewcommand{\phi}{\varphi}

\newcommand{\weg}[1]{}

\theoremstyle{definition}

\newtheorem{theorem}{Theorem}
\newtheorem{lemma}[theorem]{Lemma}
\newtheorem{definition}[theorem]{Definition}

\newtheorem{remark}[theorem]{Remark}

\newtheorem{proposition}[theorem]{Proposition}

\newtheorem{corollary}[theorem]{Corollary}
\newtheorem{fact}[theorem]{Fact}

\title{Axiomatizing Rumsfeld Ignorance}

\author{Jie Fan\\
\small Institute of Philosophy, Chinese Academy of Sciences;\\
\small School of Humanities, University of Chinese Academy of Sciences  \\
\small \texttt{jiefan@ucas.ac.cn}}
\date{}

\begin{document}

\maketitle

\begin{abstract}
    In a recent paper, Kit Fine presents some striking results concerning the logical properties of (first-order) ignorance, second-order ignorance and Rumsfeld ignorance. However, Rumsfeld ignorance is definable in terms of ignorance, which makes some existing results and the axiomatization problem trivial. A main reason is that the accessibility relations for the implicit knowledge operator contained in the packed operators of ignorance and Rumsfeld ignorance are the same. In this work, we assume the two accessibility relations to be different so that one of them is an arbitrary subset of the other. This will avoid the definability issue and retain most of the previous validities. The main results are axiomatizations over various proper bi-frame classes. Finally we apply our framework to analyze Fine's results.
\end{abstract}

\weg{\begin{abstract}
    In this paper, we propose a logic of (first-order) ignorance whether and Rumsfeld ignorance. We axiomatize this logic over various frame classes, in which the completeness of the minimal system uses a nontrivial inductive method. We then apply our work to Fine~\cite{Fine:2018}. As we will show, within our framework, we can simplify and extend some of the results in~\cite{Fine:2018}, and obtain some results that Fine did not notice. This may help us trace out more clearly the logical interactions among first-order ignorance, Rumsfeld ignorance and second-order ignorance.
\end{abstract}}


\section{Introduction}

Since Socrates~\cite{Bett:2010},
ignorance has played a pivotal role in many areas, for instance, not only in epistemology, but also in ethics, philosophy of law, social philosophy, philosophy of science and even science itself~\cite{Peels:2007}.\footnote{For instance, researchers focus on the nature/definition of ignorance, the relation between ignorance and virtue, and so on. Moreover, it is argued in~\cite{Firestein:2012} that ignorance is the engine of science.} Sometimes a reference to ignorance has featured in epistemological discussions --- as in the title of Unger~\cite{Unger:1975} --- but only as a counterpoint to knowledge, when emphasizing how little knowledge of this or that kind is possible. Recently, however, there has been an explosion of interest in ignorance in its own right, with attempts to say what exactly ignorance consists in and what its fundamental logical properties are. Instead of simply taking for granted the standard view (SV), according to which ignorance of a proposition is merely the absence of knowledge of the proposition~\cite{Unger:1975,Driver:1989,Goldmanetal:2009,LeMorvan:2010,LeMorvan:2011,LeMorvan:2012,LeMorvan:2013}, we now also have what has been called the new view (NV),
as well as what has been termed the logical view (LV) to contend with.\footnote{The terminology ‘the standard view’ is introduced in~\cite{LeMorvan:2011}, ‘the new view’ is from~\cite{Peels:2011}, whereas the term
‘the logical view’ comes from~\cite{Fan:2016}.} According to NV, ignorance is instead simply the absence of true belief~\cite{Peels:2011,Peels:2012,Kyle:2015}. According to LV, ignorance with respect to a proposition is the absence of knowledge of that proposition and of its negation~\cite{wiebeetal:2003,hoeketal:2004,steinsvold:2008,Fanetal:2014,Fanetal:2015,olsson2015explicating}.



In addition to the competing views about what ignorance consists in, various different forms of ignorance have attracted attention in the literature, such as pluralistic ignorance~\cite{o1976pluralistic,bjerring2014rationality,proietti2014ddl}, circumscriptive ignorance~\cite{konolige1982circumscriptive}, chronological ignorance~\cite{shoham1986chronological}, group ignorance~\cite{Fan:2016}, factive ignorance~\cite{kubyshkina2019logic}, radical ignorance~\cite{Fano:2020working}, relative ignorance~\cite{Goranko:2021}, disjunctive ignorance~\cite{Fan:2021disjunctive}, severe ignorance~\cite{bonzio2022logical}, and so on. In a recent paper~\cite{Fine:2018}, Kit Fine presents some striking results concerning the logical properties of ignorance, for the purpose of presenting which, Fine introduces the following terminology, all defined in terms of the primitive knowledge operator $\K$ of epistemic logic:

\begin{itemize}
    \item[(i)] One is {\em ignorant that} $\phi$, if one does not know that $\phi$ ($\neg\K\phi$)
    \item[(ii)] It is {\em (epistemically) possible} that $\phi$, if one is ignorant that not-$\phi$ ($\hK\phi=_{df}\neg\K\neg\phi$)
    \item[(iii)] One is {\em ignorant of the fact that} $\phi$, if $\phi$ is the case and one is ignorant that $\phi$ ($\phi\land\neg\K\phi$)
    \item[(iv)] One is {\em (first-order) ignorant whether $\phi$}, if one is both ignorant that $\phi$ and ignorant that not-$\phi$ ($\I\phi=_{df}\neg\K\phi\land\neg\K\neg\phi$)
    \item[(v)] One is {\em Rumsfeld ignorant of} $\phi$, if one is ignorant of the fact that one is ignorant whether $\phi$ ($\I\phi\land\neg\K\I\phi$)\footnote{As correctly observed by Fine in~\cite{Fine:2018}, Rumsfeld ignorance is a form of Fitchean ignorance; in other words, it is of the form $\phi\land\neg \K\phi$, a notion called `unknown truths' or `accident' in the literature, see e.g.~\cite{Marcos:2005,steinsvold:2008,Steinsvold:2008a}. This helps us find the suitable canonical relation for Rumsfeld ignorance in Def.~\ref{def.cm}.}
    \item[(vi)] One is {\em second-order ignorant whether} $\phi$, if one is ignorant whether one is ignorant whether $\phi$ ($\I\I\phi$)
\end{itemize}

Among other things, Fine establishes the following results within the
context of the modal system ${\bf S4}$, where (4) is surprising and interesting, with $\Box$ written as $K$ in the relevant pages (pp.~4032--4033), though in the later `formal appendix', Fine uses the $\Box$ notation (with $\Diamond$ for the above $\hK$):\footnote{Of the following, only (1) and (4) are proved, in~\cite[Lemma~3, Thm.~4]{Fine:2018}, the remained being justified by semi-formal arguments. That the result (4) is surprising and interesting, is because it disobeys our intuition. As Fine shows via an argument in~\cite[p.~4033]{Fine:2018}, in principle, one could know that those propositions, which one does not know, do not appear on the list of their knowledge; however, (4) tells us that this is impossible.}

\begin{enumerate}
    \item[(1)] Second-order ignorance implies first-order ignorance.
    \item[(2)] Second-order ignorance implies Rumsfeld ignorance.
    \item[(2$^\ast$)] Rumsfeld ignorance implies second-order ignorance.
    \item[(3)] One does not know that one is Rumsfeld ignorant.
    \item[(4)] One does not know that one is second-order ignorant.
\end{enumerate}


Now, as indicated in our opening paragraphs, the relation between knowledge and ignorance has recently become somewhat problematic, making it resonable to approach the logic of ignorance in its own right, rather than by defining ignorance --- first-order, second-order, or of the mixed `Rumsfeldian' variety --- in terms of a primitive $K$ operator, and no such operator is included here for that reason. Following Kit Fine, in the semantic treatment of ignorance (specifically first-order ignorance) and Rumsfeld ignorance below (Definition~\ref{def.semantics}, $\I$ and $\I^R$ for these notions), we will use a binary accessibility relation $R$ to interpret $\I$ and $\I^R$ exactly as if these had been defined in terms of a $\Box$ operator by the definitions $\I\phi:=\neg\Box\phi\land\neg\Box\neg\phi$ and  $\I^R\phi:=\I\phi\land\neg\Box\I\phi$.

However, as we shall show below, Rumsfeld ignorance is definable in terms of (first-order) ignorance over the class of all frames. This makes some of Kit Fine's results mentioned above and the problem of axiomatizing Rumsfeld ignorance (and (first-order) ignorance) trivial, which may be undesirable.

We can fix this issue by sticking with the ignorance-related primitives to emphasize neutrality w.r.t. the intended interpretation of this $\Box$ operator.\footnote{By ``sticking with the ignorance-related primitives'', we mean using $\I$ and $\I^R$ instead of $\Box$ as primitives. By emphasizing neutrality with respect to the interpretation of $\Box$ operator, we would like to indicate that $\Box$ can be given different interpretations in different contexts, just as the $\Box$ operators implicitly contained in the packaged operators $\I$ and $\I^R$.} For example, an adherent of SV may want to think of this as representing knowledge, and an adherent of NV may prefer to think of $\Box$ as representing true belief, while someone else again may want to think of this as representing just belief. This indicates that one may give different interpretations to the implicit $\Box$ operator in the packaged operator $\I$ on one hand, and in the packaged operator $\I^R$ on the other hand, which semantically corresponds to different accessibility relations for the implicit $\Box$ operator in $\I$ and $\I^R$, denoted $R$ and $R^{\bullet}$ below, respectively. This is what we will do. Moreover, we will assume that either $R\subseteq R^{\bullet}$ or $R^{\bullet}\subseteq R$, which may retain some validities that hold in the case of $R=R^{\bullet}$. We will show that in either case, the previous definability fails; in fact, Rumfeld ignorance cannot be defined with ignorance. We will then focus on the restriction of $R\subseteq R^{\bullet}$ (the reason of focusing on this restriction will be explained at the end of Sec.~\ref{def.semantics}) and axiomatize the logic of ignorance and Rumsfeld ignorance over various classes of frames with this restriction.

The remainder of the paper is structured as follows. Sec.~\ref{sec.synseman} introduces the syntax and semantics of the logic of (first-order) ignorance and Rumsfeld ignorance, demonstrates the definability of Rumsfeld ignorance in terms of ignorance, and make some assumptions about the relationship between the accessibility relations for the implicit $\Box$ contained in $\I$ and $\I^R$, which makes the previous definability fail. Sec.~\ref{sec.minimallogic} proposes the minimal logic of ignorance and Rumsfeld ignorance and shows its soundness and strong completeness, where the completeness is proved via a nontrivial inductive method. Sec.~\ref{sec.extensions} axiomatizes the logic over various special frame classes. Sec.~\ref{sec.applications} applies our framework to analyze Fine's results in~\cite{Fine:2018}. We conclude with some remarks in Sec.~\ref{sec.conclusion}.

\section{Syntax and Semantics}\label{sec.synseman}

Fix a nonempty set of propositional variables $\BP$. Throughout the paper, we let $p\in \BP$. For simplicity, in what follows we introduce a single-agent setting but it can be easily extended to a multi-agent case in a standard way.

\begin{definition}[Language] The language of the logic of (first-order) ignorance and Rumsfeld ignorance, denoted $\IRI$, is defined recursively as follows:
$$\phi::=p\mid \neg\phi\mid (\phi\land\phi)\mid \I\phi\mid \I^R\phi.$$
\end{definition}

Intuitively, $\I\phi$ is read ``one is {\em (first-order) ignorant whether} $\phi$'', and in the context of contingency logic, $\I\phi$ is read ``it is {\em contingent} that $\phi$''\footnote{The contingency operator, denoted $\nabla$ in the literature (see e.g.~\cite{MR66}), can be defined in terms of the necessity operator $\Box$, as $\nabla\phi=_{df}\neg\Box\phi\land\neg\Box\neg\phi$. In the epistemic setting, if $\Box$ is seen as the (propositional) knowledge operator, then $\nabla$ is (first-order) ignorance operator. Thus here is a technical relationship between contingency and ignorance. As a matter of fact, the people working on contingency logic (\cite{MR66,Humberstone95,DBLP:journals/ndjfl/Kuhn95,DBLP:journals/ndjfl/Zolin99}, etc.) and the people working on the logic of ignorance (\cite{wiebeetal:2003,hoeketal:2004,steinsvold:2008}, etc.) have been unaware of each other's work for a long time. For a survey of contingency logic, we recommend~\cite{Fanetal:2015}.}; $\I^R\phi$ is read ``one is {\em Rumsfeld ignorant of} $\phi$''. Other connectives are defined as usual; in particular, $\Kw\phi$, read ``one knows whether $\phi$'', is defined as $\neg\I\phi$. The language of the logic of ignorance, denoted $\mathcal{L}(\I)$, is the fragment of $\IRI$ without the construct $\I^R\phi$.

\weg{Fix a nonempty set of propositional variables $\BP$ and a nonempty finite set of agents $\Ag$. Throughout the paper, we let $p\in \BP$ and $a\in \Ag$.
\begin{definition}[Language] The language of the logic of (first-order) ignorance whether and Rumsfeld ignorance, denoted $\IRI$, is defined recursively as follows:
$$\phi::=p\mid \neg\phi\mid (\phi\land\phi)\mid \I_a\phi\mid \I_a^R\phi.$$
\end{definition}

Intuitively, $\I_a\phi$ is read ``the agent $a$ is {\em (first-order) ignorant whether} $\phi$'', and in the context of contingency logic, $\I_a\phi$ is read ``$\phi$ is contingent for $a$''~\cite{Fanetal:2015}; $\I^R_a\phi$ is read ``the agent $a$ is {\em Rumsfeld ignorant of} $\phi$''. Other connectives are defined as usual; in particular, $\Kw_a\phi$, read ``$a$ knows whether $\phi$'', is defined as $\neg\I_a\phi$.}

Now we follow Kit Fine~\cite{Fine:2018} to interpret the language $\IRI$. According to Fine's abbreviated definitions (see (iv) and (v) in the introduction), to interpret the meaning of Rumsfeld ignorance operator $\I^R$, one accessibility relation is enough.


A {\em (Kripke) model} is a tuple $\M=\lr{S,R,V}$, where $S$ is a nonempty set of states (or points), $R$ is a binary relation over $S$, called `accessibility relation', and $V$ is a valuation. A {\em pointed model} is a pair $(\M,s)$  where $\M$ is a model and $s$ is a state of $\M$. A {\em frame} is a model without a valuation. 

\weg{\begin{definition}[Structure]
A {\em model} is a tuple $\M=\lr{S,\{R_a\mid a\in\Ag\},V}$, where $S$ is a nonempty set of states (or points), for every $a\in \Ag$, $R_a$ is a binary relation over $S$, called `accessibility relation', and $V$ is a valuation function. A {\em pointed model} is a pair of a model with a point in it. A {\em frame} is a model without a valuation. Model $\M$ is said to be a $\mathcal{K}45$-model (resp. a $\mathcal{KD}45$-model, an $\mathcal{S}4$-model, an $\mathcal{S}5$-model), if for each $a\in\Ag$, $R_a$ is transitive and Euclidean (resp. serial and transitive and Euclidean, reflexive and transitive, reflexive and Euclidean). A $\mathcal{K}45$-frame and the like are defined similarly.
\end{definition}}


\begin{definition}[Semantics]\label{def.semantics}
Given a model $\M=\lr{S,R,V}$ and a state $s\in S$, the semantics of $\IRI$ is defined recursively as follows.
\[
\begin{array}{|lll|}
\hline
    \M,s\vDash p & \text{iff} & s\in V(p)\\
    \M,s\vDash\neg\phi & \text{iff} & \M,s\nvDash\phi\\
    \M,s\vDash\phi\land\psi&\text{iff} & \M,s\vDash\phi\text{ and }\M,s\vDash\psi\\
    \M,s\vDash\I\phi &\text{iff}& \text{for some }t\text{ such that }sRt\text{ and }\M,t\vDash\phi,\text{ and }\\
    &&\text{for some }u\text{ such that }sRu\text{ and }\M,u\nvDash\phi\\
    \M,s\vDash\I^R\phi & \text{iff}& \M,s\vDash\I\phi\text{ and for some }t\text{ such that }sRt\text{ and }\M,t\nvDash \I\phi \\
\hline
\end{array}
\]
\end{definition}

We say that $\phi$ is {\em valid on a class of frames $F$}, notation: $F\vDash\phi$, if for all frames $\mathcal{F}$ in $F$, for all models $\M$ based on $\mathcal{F}$, for all $s$ in $\M$, we have $\M,s\vDash\phi$. We say that $\phi$ is {\em valid}, notation: $\vDash\phi$, if $\phi$ is valid on the class of all frames. Given $\Gamma\subseteq\IRI$, we say that $\M,s\vDash\Gamma$, if for all $\phi\in\Gamma$ we have $\M,s\vDash\phi$. We say $\phi$ is {\em satisfiable}, if there exists a pointed model $(\M,s)$ such that $\M,s\vDash\phi$. We say that $\phi$ is a {\em logical consequence} of $\Gamma$ over a class of frames $F$, notation: $\Gamma\vDash_F\phi$, if for all frames $\mathcal{F}$, for all models $\M$ based on $\mathcal{F}$, for all $s$ in $\M$, if $\M,s\vDash\Gamma$, then $\M,s\vDash\phi$.

Under this semantics, we have the following results which describe the interactions beween first-order ignorance and Rumsfeld ignorance. The proof is similar to that of Prop.~\ref{prop.bimodel-validities} below.

\begin{proposition}\label{prop.validities-semantics} Let $\phi,\psi,\chi\in\IRI$ be arbitrary.
\begin{itemize}
\item[(i)] $\vDash\I\phi\land \I(\I\phi\vee\chi)\to\I^R\phi$;
\item[(ii)] If $\vDash\psi\to\I\phi$, then $\vDash\I^R\phi\to\I\psi\lor\I(\psi\vee\phi)$.
\end{itemize}
\end{proposition}

However, it may be surprising to see that under the semantics, Rumsfeld ignorance is definable in terms of first-order ignorance.
\begin{proposition}\label{prop.definable-ribyi} Let $\phi\in\mathcal{L}(\I)$ be arbitrary. We have\footnote{This result is credited to Yanjing Wang.\label{fn.yanjing}}
$$\vDash\I^R\phi\lra\I\phi\land (\I\I\phi\vee\I(\phi\to\I\phi)).$$
\end{proposition}

\begin{proof}
Let $\M=\lr{S,R,V}$ be an arbitrary model and $s\in S$. Suppose that $\M,s\vDash \I^R\phi$. By semantics of $\I^R\phi$, it is obvious that $\M,s\vDash\I\phi$ and there exists $t$ such that $sRt$ and $\M,t\nvDash\I\phi$. From $\M,s\vDash\I\phi$, it follows that there are $u,v$ such that $sRu$, $sRv$ and $\M,u\vDash\phi$ and $\M,v\nvDash\phi$. Assume towards contradiction that $\M,s\nvDash\I\I\phi\vee\I(\phi\to\I\phi)$. Since $\M,v\nvDash\phi$, we have $\M,v\vDash\phi\to\I\phi$. Using this and $s\nvDash\I(\phi\to\I\phi)$ and $sRu$ and $sRv$, we infer that $\M,u\vDash\phi\to\I\phi$, thus $\M,u\vDash\I\phi$. From this and $\M,t\nvDash\I\phi$ and $sRt$ and $sRu$, we obtain that $\M,s\vDash\I\I\phi$: a contradiction.

Conversely, suppose that $\M,s\vDash\I\phi\land (\I\I\phi\vee\I(\phi\to\I\phi))$, to show that $\M,s\vDash\I^R\phi$. By semantics, it suffices to show that there is a $t$ such that $sRt$ and $\M,t\nvDash\I\phi$. We consider two cases. If $\M,s\vDash \I\I\phi$, by semantics of $\I$, we are done; if $\M,s\vDash\I(\phi\to\I\phi)$, then there exists $t$ such that $sRt$ and $\M,t\nvDash\phi\to\I\phi$, and so $\M,t\nvDash\I\phi$, as desired.
\end{proof}

Note that this is a surprising result. As one may know, the expressive power of $\mathcal{L}(\I)$ is very weak~(see~\cite{Fanetal:2015}, the metaphysical counterpart of ignorance is contingency. This is
similar to the case of (propositional) knowledge and necessity: the metaphysical
counterpart of knowledge is necessity.). Even so, it can define the complex notion of Rumsfeld ignorance.

Due to the definability result, some results about the relationships among first-order ignorance, second-order ignorance and Rumsfeld ignorance in~\cite{Fine:2018}, as mentioned in the introduction, and the axiomatization problem become trivial.\footnote{Since the axiomatizations of the logic of ignorance already exist in the literature (see e.g.~\cite{hoeketal:2004,MR66,Humberstone95,DBLP:journals/ndjfl/Kuhn95,DBLP:journals/ndjfl/Zolin99,Fanetal:2014,Fanetal:2015}), thus the axiomatizations of Rumsfeld ignorance and ignorance can be obtained from a translation induced by the definability result. Moreover, this also makes some of Kit Fine's results mentioned above trivial. Take the item (2) ``Second-order ignorance implies Rumsfeld ignorance'' mentioned in the introduction as an example. By Prop.~\ref{prop.definable-ribyi} and item (1) ``Second-order ignorance implies first-order ignorance'', it should be trivial to obtain the item (2).} In retrospect, one main reason for the definability result is that the accessibility relations for the implicit $\Box$ operator contained in the packaged operators $\I$ and $\I^R$ are the same.\footnote{Note that this does {\em not} mean that the $\Box$ operator can be defined with the operators $\I$ and $\I^R$. Rather, it means that $\Box$ is implicitly contained in the packaged operators $I$ and $I^R$; for instance, since $I\phi$ abbreviates $\neg\Box\phi\land\neg\Box\neg\phi$, the operator $I$ packages the connectives $\neg$, $\Box$ and $\land$.} 

To address the above issue, we assume the two accessibility relations to be such that one of them is an arbitrary subset of the other. As we will check, this will deal with the definability issue but retain many validities in the case when the two accessibility relations are the same.

We say that $\M=\lr{S,R,R^{\bullet},V}$ is a {\em bi-model}, if $S$ is a nonempty set of states (or points), $V$ is a valuation, and $R$ and $R^{\bullet}$ are accessibility relations for the implicit $\Box$ operator contained in $\I$ and $\I^R$, respectively. A {\em bi-frame} is a bi-model without valuations. A bi-frame $\mathcal{F}=\lr{S,R,R^\bullet}$ is a $P$ bi-frame, if $R$ and $R^\bullet$ both satisfy the property $P$; a bi-model $\M$ is a $P$ bi-model, if its underlying frame is a $P$ bi-frame. Following Fine~\cite{Fine:2018}, in the sequel we will consider additional properties such as reflexivity and transitivity. To distinguish the new semantics from the previous semantics, we call our semantics `bi-semantics' correspondingly. In what follows, we will focus on the bi-semantics.

\begin{definition}[Bi-semantics]\label{def.semantics-new}
Given a bi-model $\M=\lr{S,R,R^{\bullet},V}$ and a state $s\in S$, the bi-semantics of $\IRI$ is defined recursively as follows.
\[
\begin{array}{|lll|}
\hline
    \M,s\vDash p & \text{iff} & s\in V(p)\\
    \M,s\vDash\neg\phi & \text{iff} & \M,s\nvDash\phi\\
    \M,s\vDash\phi\land\psi&\text{iff} & \M,s\vDash\phi\text{ and }\M,s\vDash\psi\\
    \M,s\vDash\I\phi &\text{iff}& \text{for some }t\text{ such that }sRt\text{ and }\M,t\vDash\phi,\text{ and }\\
    &&\text{for some }u\text{ such that }sRu\text{ and }\M,u\nvDash\phi\\
    \M,s\vDash\I^R\phi & \text{iff}& \M,s\vDash\I\phi\text{ and for some }t\text{ such that }sR^{\bullet}t\text{ and }\M,t\nvDash \I\phi \\
\hline
\end{array}
\]
\end{definition}

We say that $\phi$ is {\em valid on a class of bi-frames $F$}, notation: $F\vDash\phi$, if for all bi-frames $\mathcal{F}$ in $F$, for all bi-models $\M$ based on $\mathcal{F}$, for all $s$ in $\M$, we have $\M,s\vDash\phi$. We say that $\phi$ is valid under the bi-semantics, if $\phi$ is valid on the class of all bi-frames. Given $\Gamma\subseteq\IRI$, we say that $\M,s\vDash\Gamma$, if for all $\phi\in\Gamma$ we have $\M,s\vDash\phi$. We say $\phi$ is {\em satisfiable} under the bi-semantics, if there exists a pointed bi-model $(\M,s)$ such that $\M,s\vDash\phi$. We say that $\phi$ is a {\em logical consequence} of $\Gamma$ over a class of bi-frames $F$, notation: $\Gamma\vDash_F\phi$, if for all bi-frames $\mathcal{F}$, for all bi-models $\M$ based on $\mathcal{F}$, for all $s$ in $\M$, if $\M,s\vDash\Gamma$, then $\M,s\vDash\phi$.

Recall that under the previous standard semantics (namely, Def.~\ref{def.semantics}), $\I\phi\land \I(\I\phi\vee\chi)\to\I^R\phi$ is valid, and $\dfrac{\psi\to\I\phi}{\I^R\phi\to\I\psi\lor\I(\psi\vee\phi)}$ is validity-preserving (Prop.~\ref{prop.validities-semantics}). In contrast, this is {\em not} the case under the bi-semantics.

\begin{proposition}\label{prop.bimodel-validities} \ 
\begin{itemize}
\item[(i)] On the class of all bi-frames satisfying the condition $R\subseteq R^{\bullet}$, $\I\phi\land \I(\I\phi\vee\chi)\to\I^R\phi$ is valid, but $\dfrac{\psi\to\I\phi}{\I^R\phi\to\I\psi\lor\I(\psi\vee\phi)}$ is {\em not} validity-preserving. 
\item[(ii)] On the class of all bi-frames satisfying the condition $R^{\bullet}\subseteq R$, $\dfrac{\psi\to\I\phi}{\I^R\phi\to\I\psi\lor\I(\psi\vee\phi)}$ is validity-preserving, but $\I\phi\land \I(\I\phi\vee\chi)\to\I^R\phi$ is {\em not} valid. 
\end{itemize}
\end{proposition}

\begin{proof}
let $\M=\lr{S,R,R^{\bullet},V}$ be an arbitrary bi-model and $s\in S$.

For (i), assume that $R\subseteq R^{\bullet}$. Suppose that $\M,s\vDash\I\phi\land \I(\I\phi\vee\chi)$, to show that $\M,s\vDash\I^R\phi$. By supposition that $\M,s\vDash\I(\I\phi\vee\chi)$, there exists $t$ such that $sRt$ and $t\nvDash\I\phi\vee\chi$, thus $t\nvDash\I\phi$. By assumption and $sRt$, we have $sR^{\bullet}t$. Therefore, $\M,s\vDash\I^R\phi$.

For the other part, since $\vDash\I\phi\to\I\phi$, it suffices to construct a bi-model such that $R\subseteq R^{\bullet}$ and $\nvDash\I^R\phi\to\I\I\phi\vee\I(\I\phi\vee\phi)$. Consider the following bi-model, where $R\subseteq R^\bullet$, and $x:\phi$ means that the world named $x$ satisfies the formula $\phi$. For simplicity, we omit the reference to $R^\bullet$, which holds on all arrows.
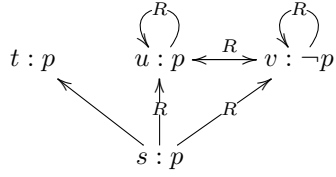
\weg{\[
\xymatrix{&&&t:p\\
\M&s:p\ar[urr]\ar[rr]|{R}\ar[drr]|{R} && u:p\ar[d]\ar@(ur,ul)|{R}\\
&&& v:\neg p\ar[u]_{R}\ar@(dl,dr)|{R}}
\]}
\[
\xymatrix{t:p&u:p\ar@(ur,ul)|{R}\ar[r]^R&v:\neg p\ar@(ur,ul)|{R}\ar[l]\\
&s:p\ar[ul]\ar[u]|R\ar[ur]|R&\\
}
\]
We can show that $\M,s\nvDash\I^Rp\to\I\I p\vee\I(\I p\vee p)$, as follows.
\begin{itemize}
\item $s\vDash\I^Rp$: since $u\vDash p$ and $v\nvDash p$ and $sRu,sRv$, we have $s\vDash\I p$; moreover, $sR^{\bullet}t$ and $t\nvDash\I p$.
\item $s\nvDash\I\I p$: this is because $R(s)=\{u,v\}$ and $u\vDash\I p$ and $v\vDash \I p$.
\item $s\nvDash\I(\I p\vee p)$: follows from $R(s)=\{u,v\}$ and $u\vDash \I p\vee p$ (as $u\vDash p$) and $v\vDash \I p\vee p$ (as $v\vDash \I p$).
\end{itemize}

For (ii), assume that $R^{\bullet}\subseteq R$ and $\vDash\psi\to\I\phi$, to show that $\vDash\I^R\phi\to\I\psi\lor\I(\psi\vee\phi)$. For this, suppose that $\M,s\vDash\I^R\phi$, it remains only to prove that $\M,s\vDash \I\psi\lor\I(\psi\vee\phi)$. If not, that is, $\M,s\nvDash\I\psi\lor\I(\psi\vee\phi)$, then by supposition, we infer that $\M,s\vDash\I\phi$ and for some $t$ such that $sR^{\bullet}t$ and $\M,t\nvDash\I\phi$. Since $s\vDash\I\phi$, there are $u$ and $v$ such that $sRu$ and $sRv$ and $\M,u\vDash\phi$ (thus $u\vDash\psi\vee\phi$) and $\M,v\nvDash\phi$. From $s\nvDash\I(\psi\vee\phi)$, $sRu$, $sRv$ and $\M,u\vDash\psi\vee\phi$, it follows that $\M,v\vDash\psi\vee\phi$, thus $\M,v\vDash\psi$. As $R^{\bullet}\subseteq R$ and $sR^{\bullet}t$, we get $sRt$. Due to the fact that $s\nvDash\I\psi$, we can derive that $\M,t\vDash\psi$, and thus $\M,t\nvDash\psi\to\I\phi$, contrary to the assumption.

For the invalidity part, consider the following bi-model, where $R^{\bullet}=\emptyset$ and the valuation at $x$ is arbitrary.
\weg{\[
\xymatrix{&&&y:p\\
\M&x\ar[urr]|R\ar[rr]|R&&z:\neg p\\}
\]}
\[
\xymatrix{
\M&y:p&&x\ar[ll]|R\ar[rr]|R&&z:\neg p\\}
\]
One may check that $\M,x\vDash\I p\land\I(\I p\vee p)$ but $x\nvDash\I^Rp$.
\end{proof}

\begin{remark}
Note that the invalidity in Prop.~\ref{prop.bimodel-validities}(ii) still holds even if $R^\bullet$ is serial. Consider the following bi-model, where we omit the reference to $R$, which holds on all arrows:
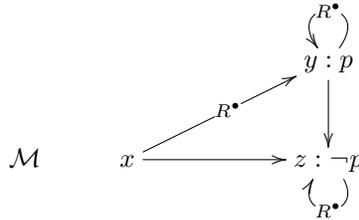
\[
\xymatrix{&&&y:p\ar@(ur,ul)|{R^\bullet}\ar[d]\\
\M&x\ar[urr]|{R^\bullet}\ar[rr]&&z:\neg p\ar@(dr,dl)|{R^\bullet}\\}
\]
One may easily verify that $\M,x\vDash \I p\land \I(\I p\vee p)$ but $x\nvDash\I^R p$.
\end{remark}

We have shown in Prop.~\ref{prop.definable-ribyi} that Rumsfeld ignorance is definable with ignorance whenever $R=R^{\bullet}$. It is natural to ask if Rumsfeld ignorance can be defined with ignorance in that way if either $R\subseteq R^{\bullet}$ or $R^{\bullet}\subseteq R$. The answer would be negative. In fact, when it is the case of either $R\subseteq R^{\bullet}$ or $R^{\bullet}\subseteq R$, Rumsfeld ignorance is undefinable with ignorance in any way.

In what follows, to show an operator $O$ is undefinable in a language $\mathcal{L}$ over a certain class of frames $F$, it suffices to find two pointed models based on frames in $F$, such that $\mathcal{L}$ cannot distinguish between the two pointed models but there exists a formula $\phi$ in $\mathcal{L}$ such that $O\phi$ can distinguish between them.

\begin{proposition}
$\I^R$ is not definable in $\mathcal{L}(\I)$ over the class of all bi-frames satisfying the condition $R\subseteq R^{\bullet}$.
\end{proposition}

\begin{proof}
Consider the following bi-models, where $R\subseteq R^\bullet$. For simplicity, we omit the reference to $R^\bullet$, which holds on all arrows.
$$
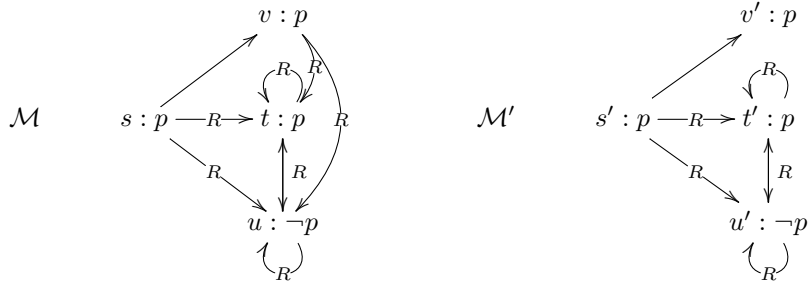
\xymatrix{&& v:p\ar@/^12pt/[d]|{R}\ar@/^22pt/[dd]|{R}&&&&v':p\\
\M& s:p\ar[ur]\ar[r]|{R}\ar[dr]|{R} & t:p\ar[d]^{R}\ar@(ur,ul)|{R}&&\M'&s':p\ar[ur]\ar[r]|{R}\ar[dr]|{R} & t':p\ar[d]^{R}\ar@(ur,ul)|{R}&\\
&& u:\neg p\ar[u]\ar@(dr,dl)|{R}&&&& u':\neg p\ar[u]\ar@(dr,dl)|{R}\\}
$$

Define $Z=\{(s,s'),(t,t'),(u,u')\}$. We can show that $Z$ is a $\I$-bisimulation, called `$\Delta$-bisimulation' in~\cite[Def.~3.3]{Fanetal:2014}.\footnote{Let $\M=\lr{S,R,V}$ be a model. A binary relation $Z$ is a {\em $\Delta$-bisimulation} on $\M$, if $Z$ is nonempty and whenever $sZs'$:
\begin{itemize}
    \item[(Var)] $s$ and $s'$ satisfy the same propositional variables;
    \item[($\Delta$-Zig)] if there are $t_1,t_2$ such that $sRt_1$ and $sRt_2$ and $(t_1,t_2)\notin Z$ and $sRt$ for some $t$, then there is a $t'$ such that $s'Rt'$ and $(t,t')\in Z$;
    \item[($\Delta$-Zag)] if there are $t_1',t_2'$ such that $s'Rt_1'$ and $s'Rt_2'$ and $(t_1',t_2')\notin Z$ and $s'Rt'$ for some $t'$, then there is a $t$ such that $sRt$ and $(t,t')\in Z$.
\end{itemize}
We say that $(\M,s)$ and $(\M',s')$ are {\em $\Delta$-bisimilar}, if there is a $\Delta$-bisimulation $Z$ in the disjoint union of $\M$ and $\M'$ such that $(s,s')\in Z$. It is shown in~\cite[Prop.~3.5]{Fanetal:2014} that $\mathcal{L}(I)$ is invariant under $\Delta$-bisimulation.} Then by ~\cite[Prop.~3.5]{Fanetal:2014}, $\mathcal{L}(I)$ cannot be distinguished by $(\M,s)$ and $(\M',s')$. 

On the other hand, $\I^Rp$ can distinguish between both pointed bi-models. To see this, note that $\M\vDash\I p$, which implies that $\M,s\nvDash \I^Rp$; however, since $s'\vDash \I p$ and $s'R^{\bullet}v'$ and $v'\nvDash\I p$, thus $\M',s'\vDash\I^Rp$.
\end{proof}

It is worth noting that if we take the reflexive closure of $\M$ and $\M'$ in the above proof, the undefinability result still holds. This means that $I^R$ is not definable in $\mathcal{L}(\I)$ over the class of all reflexive bi-frames satisfying the condition $R\subseteq R^\bullet$.

\begin{proposition}
$\I^R$ is not definable in $\mathcal{L}(\I)$ over the class of all bi-frames satisfying the condition $R^{\bullet}\subseteq R$.
\end{proposition}

\begin{proof}
Consider the following bi-models, where $R^{\bullet}\subseteq R$. For simplicity, we omit the reference to $R$, which holds on all arrows.
$$
\xymatrix{\M_1& t_1:p & s_1:p\ar[r]\ar[l] & u_1:\neg p\\
\M_2& t_2:p & s_2:p\ar[r]^{R^{\bullet}}\ar[l] & u_2:\neg p\\}
$$

Define $Z=\{(s_1,s_2),(t_1,t_2),(u_1,u_2)\}$. One may check that $Z$ is a $\I$-bisimulation, called `$\Delta$-bisimulation' in~\cite[Def.~3.3]{Fanetal:2014}. Then by ~\cite[Prop.~3.5]{Fanetal:2014}, $\mathcal{L}(I)$ cannot be distinguished by $(\M_1,s_1)$ and $(\M_2,s_2)$. However, these two pointed bi-models can be distinguished by a $\I^R$-formula $\I^Rp$, since $\M_1,s_1\nvDash\I^Rp$ and $\M_2,s_2\vDash\I^Rp$, as desired.
\end{proof}

Again, as one may check, if we take the reflexive closure of $\M$ and $\M'$ in the above proof, the undefinability result still holds. This means that $\I^R$ is not definable in $\mathcal{L}(\I)$ over the class of all reflexive bi-frames satisfying the condition $R^\bullet\subseteq R$.

\weg{Note that in the semantics of $\I^R_a\phi$, although $\I^R_a$ is `defined' with $\I_a$, $\I_a\phi$ is {\em not} a subformula of $\I^R_a\phi$. This brings about a technical difficulty in the completeness proof, as we shall see below.

For the sake of reference, we also introduce the operator $\Box_a$ mentioned in the introduction and interpret it as follows:
\[
\begin{array}{lll}
    \M,s\vDash\Box_a\phi & \text{iff} & \text{for all }t\text{ such that }sR_at,\M,t\vDash\phi. \\
     &
\end{array}
\]

Let $a\in\Ag$. If $R_a\subseteq R_a^K$, then $\vDash\neg\I^R_a\phi\land\I_a\phi\to\neg\I_a(\I_a\phi\vee\chi)$; if $R_a^K\subseteq R_a$, then $\vDash\psi\to\I_a\phi$ implies $\vDash\neg\I_a\psi\land\neg\I_a(\psi\vee\phi)\to\neg\I_a^R\phi$.}

\weg{\section{Expressivity}

\begin{proposition}
$\IRI$ is more expressive than $\LI$.
\end{proposition}

\begin{proof}
Consider the following models:
\[
\xymatrix{s:p\ar[r]\ar[dr]&u:p\\
&v:\neg p\\
&\M}
\qquad
\qquad
\xymatrix{s':p\ar[r]\ar[dr] & u':p\ar[r]\ar[dr] & x':p\\
& v':\neg  p\ar[r]\ar[ur] & y':\neg p\\
&\M'&}
\]

\end{proof}
}

In what follows, we focus on the restriction of $R\subseteq R^{\bullet}$, and call those bi-models (bi-frames) with this restriction `proper bi-models (resp. proper bi-frames)'. Moreover, we say that such a bi-model has a property, if both $R$ and $R^{\bullet}$ have that property. Such a restriction is based on the following reason. The notion of Rumsfeld ignorance, proposed in~\cite{Fine:2018}, is supposed to formalize the 2002 remark about {\em unknown unknowns} by the US Defense Secretary Donald Rumsfeld, ``But there are also unknown unknowns. These are things we don't know we don't know'', which is a counterexample to the usual axiom 5; in other words, one does not know something but one does not know {\em that} one does not know something.\footnote{In fact, before Kit Fine~\cite{Fine:2018}, Humberstone~\cite[p.~371]{Humberstone:2016} has a brief discussion about the remark of Donald Rumsfeld.}
Kit Fine uses `knowing whether' (i.e. the negation of (first-order) ignorance) to express this `know something', and uses `knowing that' (that is, propositional knowledge) to express this `know that'. Moreover, as Hintikka noted in~\cite{hintikka:1962}, `knowing whether'
is a notion weaker than `knowing that'. This corresponds to the semantic restriction that $R\subseteq R^\bullet$, where $R$ and $R^\bullet$ are, respectively, the accessibility relations for `knowing whether'/(first-order) ignorance and `knowing that'. 

\section{Minimal Logic}\label{sec.minimallogic}

This section proposes the minimal logic of $\IRI$ and shows its soundness with respect to the class of proper bi-frames.

\subsection{Proof system and Soundness}
\begin{definition}\label{def.minimallogic}
The minimal logic of $\IRI$, denoted $\mathbb{IRIK}$, consists of the following axioms and inference rules, where $n\in\mathbb{N}$.
\[
\begin{array}{ll}
    \text{TAUT} & \text{all instances of propositional tautologies} \\
    \text{I-Equ} & \I\phi\lra \I\neg\phi\\
    \text{IR-Equ} & \I^R\phi\lra \I^R\neg\phi\\
    \text{I-Con} & \I(\phi\land\psi)\to (\I\phi\vee\I\psi)\\
    \text{I-Dis} & \I(\phi\vee\psi)\land\I(\neg\phi\vee\chi)\to\I\phi\\
    \text{RI-I} & \I^R\phi\to \I\phi\\
    \text{MIX} & \I\phi\land \I (\I\phi\vee\chi)\to \I^R\phi\\
    \text{R-NI} & \dfrac{\phi}{\neg\I\phi}\\
    \text{RE-I} & \dfrac{\phi\lra \psi}{\I\phi\lra \I\psi}\\
    \text{RE-RI} & \dfrac{\phi\lra \psi}{\I^R\phi\lra \I^R\psi}\\
    \text{R-MIX}&\dfrac{\I\chi_1\land\cdots\land\I\chi_n\to\I\phi}{(\neg \I^R\chi_1\land \I\chi_1)\land\cdots\land (\neg \I^R\chi_n\land \I\chi_n)\to\neg \I^R\phi}\\
\end{array}
\]
\end{definition}





The axioms and inference rules concerning the operator $\I$ only are equivalent transformations of the corresponding ones concerning the operator $\Kw$, also denoted $\Delta$, in the literature, see e.g.~\cite{DBLP:journals/ndjfl/Kuhn95,Fanetal:2015}. Moreover, $\text{RI-I}$, $\text{MIX}$, and $\text{R-MIX}$ concern the interactions between Rumsfeld ignorance and first-order ignorance: $\text{RI-I}$ says that Rumsfeld ignorance implies first-order ignorance, $\text{MIX}$ tells us how to derive Rumsfeld ignorance; in particular, by letting $\chi$ be $\bot$, we obtain that ``first-order ignorance plus second-order ignorance implies Rumsfeld ignorance'', whereas R-MIX essentially says that if one is ignorant whether finitely many formulas implies  one is ignorant whether $\phi$, then knowing the ignorance of these formulas implies one is not Rumsfeld ignorance of $\phi$.

The notions of theorems, provability, and derivation are defined as usual.

\begin{proposition}\label{prop.soundness-k}
$\mathbb{IRIK}$ is sound with respect to the class of all proper bi-frames. \end{proposition}

\begin{proof}
The validity of the axioms and inference rules only concerning the operator $\I$ can be found in~\cite[Prop.~4.2]{Fanetal:2015}. The validity of the axiom $\text{RI-I}$ is direct from the semantics of $\I^R$. The validity of the axiom $\text{MIX}$ can be found in Prop.~\ref{prop.bimodel-validities}(i). Moreover, it is not hard to show that the rule R-MIX preserves validity.
\end{proof}

\weg{\begin{proof}
The validity of the axioms and inference rules only concerning the operator $\I_a$ can be found in~\cite[Prop.~4.2]{Fanetal:2015}. The validity of the axiom $\text{RI-I}$ is direct from the semantics of $\I^R_a$. It remains only to prove the validity of the axiom $\text{MIX}$ and the inference rule $\text{R-MIX}$. For this, let $\M=\lr{S,\{R_a\mid a\in\Ag\},V}$ be an arbitrary model and $s\in S$.

Suppose that $\M,s\vDash\neg\I^R_a\phi\land\I_a\phi$, to show that $\M,s\vDash\neg\I_a(\I_a\phi\vee\chi)$. By supposition, we have for all $t$ such that $sR_at$, $\M,t\vDash I_a\phi$, thus $\M,t\vDash \I_a\phi\vee\chi$. Therefore, $\M,s\vDash\neg\I_a(\I_a\phi\vee\chi)$.

Assume that $\vDash\psi\to\I_a\phi$, to show that $\vDash\neg\I_a\psi\land \neg\I_a(\psi\vee\phi)\to\neg\I^R_a\phi$. Suppose that $\M,s\vDash \neg\I_a\psi\land \neg\I_a(\psi\vee\phi)$, it remains only to prove that $\M,s\vDash \neg\I^R_a\phi$. If not, that is, $\M,s\vDash\I^R_a\phi$, then $\M,s\vDash\I_a\phi$ and for some $t$ such that $sR_at$ and $\M,t\nvDash\I_a\phi$. From $\M,s\vDash\I_a\phi$, it follows that there are $u$ and $v$ such that $sR_au$ and $sR_av$ and $\M,u\vDash\phi$ and $\M,v\nvDash\phi$. Since $\M,u\vDash\phi$, we have $\M,u\vDash\psi\vee\phi$. This plus $\M,s\vDash\neg\I_a(\psi\vee\phi)$ implies that $\M,v\vDash \psi\vee\phi$, and thus $\M,v\vDash \psi$. Using $\M,s\vDash\neg\I_a\psi$, we can obtain that $\M,t\vDash\psi$, and then $\M,t\nvDash\psi\to\I_a\phi$, which contradicts the assumption. Therefore, $\M,s\vDash \neg\I^R_a\phi$, as desired.
\end{proof}}

\subsection{Completeness}

To show the completeness of $\mathbb{IRIK}$, we adopt a standard method of a canonical model construction, where we need to define suitable canonical relations. In the proof of the truth lemma, we need to consider two nontrivial cases $\I\phi$ and $\I^R\phi$. The hardest part is the case $\I^R\phi$, which involves the iteration of modalities. Instead, we treat $\I\phi$ in the semantics of $\I^R\phi$ `as a whole', but then we cannot use the inductive method on the {\em structure} of formulas, since $\I\phi$ is {\em not} a subformula of $\I^R\phi$. To solve this problem, we adopt an inductive method from~\cite{APAL2015}, which is also used in~\cite{ditmarschfan:2016,Fan:2016}.

To begin with, we introduce some notions. Intuitively, the size of a formula is about the number of propositional variables and primitive connectives contained in a formula.

\weg{\begin{definition}[Size, $\I^R$-depth] The {\em size} and {\em $\I^R$-depth} of $\IRI$-formulas $\phi$, denoted $S(\phi)$ and $d(\phi)$ respectively, are defined recursively as follows.
\[
\begin{array}{lllllll}
    S(p) & = & 1 && d(p) & = & 0 \\
    S(\neg\phi) & = & 1+S(\phi) && d(\neg\phi) & = & d(\phi)\\
    S(\phi\land\psi)&=&S(\phi)+S(\psi)+1 && d(\phi\land\psi)&=& \text{max}\{d(\phi),d(\psi)\}\\
    S(\I\phi)&=&1+S(\phi) && d(\I\phi)&=&d(\phi)\\
    S(\I^R\phi)&=&1+S(\phi) && d(\I^R\phi)&=&1+d(\phi)\\
\end{array}
\]
\end{definition}}

\begin{definition}[Size] The {\em size} of $\IRI$-formulas $\phi$, denoted $S(\phi)$, is defined recursively to be a positive integer as follows.
\[
\begin{array}{lll}
    S(p) & = & 1 \\
    S(\neg\phi) & = & 1+S(\phi)\\
    S(\phi\land\psi)&=&S(\phi)+S(\psi)+1\\
    S(\I\phi)&=&1+S(\phi)\\
    S(\I^R\phi)&=&1+S(\phi)\\
\end{array}
\]
\end{definition}

Intuitively, the $I^R$-depth of a formula is about the number of $I^R$ contained in the formula.
\begin{definition}[$\I^R$-depth] The {\em $\I^R$-depth} of $\IRI$-formulas $\phi$, denoted $d(\phi)$, is defined recursively to be a non-negative integer as follows.
\[
\begin{array}{lll}
    d(p) & = & 0 \\
    d(\neg\phi) & = & d(\phi)\\
    d(\phi\land\psi)&=& \text{max}\{d(\phi),d(\psi)\}\\
    d(\I\phi)&=&d(\phi)\\
    d(\I^R\phi)&=&1+d(\phi)\\
\end{array}
\]
\end{definition}

\begin{definition}[$<^S_d$]
We define the binary relation $<^S_d$ between formulas in the following way:
$$
\begin{array}{lll}
    \phi<^S_d\psi & \text{iff} & \text{either } d(\phi)<d(\psi),\text{ or } d(\phi)=d(\psi)\text{ and }S(\phi)<S(\psi).\\
     &
\end{array}$$
\end{definition}

It is obvious that $<^S_d$ is a well-founded strict partial order between formulas. Moreover, the following result is straightforward by definitions.
\begin{proposition}\label{prop.sd-induction}
Let $\phi,\psi\in\IRI$. Then
\[
\begin{array}{lllll}
    (a) & \phi<^S_d\neg\phi & & (b) & \phi<^S_d\phi\land\psi,~\psi<^S_d\phi\land\psi \\
    (c) & \phi<^S_d\I\phi && (d) & \I\phi<^S_d\I^R\phi\\
\end{array}
\]
\weg{\begin{itemize}
    \item[(a)] $\phi<^S_d\neg\phi$
    \item[(b)] $\phi<^S_d\phi\land\psi$, $\psi<^S_d\phi\land\psi$
    \item[(c)] $\phi<^S_d\I\phi$
    \item[(d)] $\I\phi<^S_d\I^R\phi$
\end{itemize}}
\end{proposition}

Now we are ready to construct the canonical model for $\mathbb{IRIK}$, where $R^c$ is defined as in~\cite{DBLP:journals/ndjfl/Kuhn95}, and $R^{\bullet c}$ is inspired by the canonical relation in the minimal accident logic (see e.g.~\cite{Steinsvold:2008a}).
\begin{definition}\label{def.cm}
$\M^c=\lr{S^c,R^c,R^{\bullet c},V^c}$ is the {\em canonical model} for $\mathbb{IRIK}$, if
\begin{itemize}
\item $S^c=\{s\mid s\text{ is a maximal consistent set of formulas for }\mathbb{IRIK}\}$.
\item $sR^ct$ iff for all $\phi$, if $\neg\I(\phi\vee\psi)\in s$ for all $\psi$, then $\phi\in t$.
\item $sR^{\bullet c}t$ iff for all $\I\phi$, if $\neg\I^R\phi\land\I\phi\in s$, then $\I\phi\in t$.
\item for each $p\in \BP$, $V^c(p)=\{s\in S^c\mid p\in s\}$.
\end{itemize}
\end{definition}

\begin{proposition}
$\M^c$ is a proper model.
\end{proposition}

\begin{proof}
It suffices to show that $R^c\subseteq R^{\bullet c}$.

Suppose that $sR^ct$, to show that $sR^{\bullet c}t$. For this, assume for an arbitrary $\I\phi$ that $\neg\I^R\phi\land\I\phi\in s$. By axiom \text{MIX}, we obtain that for each $\psi$, $\neg\I(\I\phi\vee\psi)\in s$. Then using supposition, we infer that $\I\phi\in t$, as desired.
\end{proof}

Lindenbaum's Lemma can be shown in the standard way.
\begin{lemma}[Lindenbaum's Lemma]
Every $\mathbb{IRIK}$-consistent set of formulas can be extended to a maximal $\mathbb{IRIK}$-consistent set.
\end{lemma}

\begin{lemma}[Truth Lemma]\label{lem.truthlemma-k}
For all $s\in S^c$, for all $\phi\in \IRI$, we have
$$\M^c,s\vDash\phi\iff \phi\in s.$$
\end{lemma}

\begin{proof}
By $<^S_d$-induction on $\phi$, where the only nontrivial cases are $\I\phi$ and $\I^R\phi$. The case $\I\phi$ is shown as in~\cite[Lemma~2]{DBLP:journals/ndjfl/Kuhn95} (note that $\phi<^S_d\I\phi$ by Prop.~\ref{prop.sd-induction}(c)). It remains only to treat the case $\I^R\phi$. By Prop.~\ref{prop.sd-induction}(d), $\I\phi<^S_d\I^R\phi$.

Suppose that $\I^R\phi\in s$, to prove that $\M^c,s\vDash \I^R\phi$, which by induction hypothesis amounts to saying that $\I\phi\in s$ and for some $t\in S^c$, $sR^{\bullet c}t$ and $\I\phi\notin t$. $\I\phi\in s$ is straightforward by supposition and axiom $\text{RI-I}$. In what follows, we show that
\begin{center}
$(\ast)$ $\{\I\chi\mid \neg \I^R\chi\land \I\chi\in s\}\cup\{\neg \I\phi\}$ is consistent.
\end{center}

If $(\ast)$ does not hold, then there are $\chi_1$, $\cdots$, $\chi_n$ such that $\neg \I^R\chi_m\land \I\chi_m\in s$ for all $1\leq m\leq n$ and
$$\vdash \I\chi_1\land\cdots\land\I\chi_n\to\I\phi.$$
Using the rule R-MIX we obtain
$$\vdash (\neg \I^R\chi_1\land \I\chi_1)\land\cdots\land (\neg \I^R\chi_n\land \I\chi_n)\to\neg \I^R\phi.$$
Since $\neg \I^R\chi_m\land \I\chi_m\in s$ for all $1\leq m\leq n$, we derive that $\neg \I^R\phi\in s$, contrary to the supposition.

We have thus shown $(\ast)$. Then by Lindenbaum's Lemma, there is a $t\in S^c$ such that $\{\I\chi\mid \neg \I^R\chi\land \I\chi\in s\}\cup\{\neg \I\phi\}\subseteq t$. By definition of $R^{\bullet c}$ and the consistency of $t$, we conclude that $sR^{\bullet c}t$ and $\I\phi\notin t$.

Conversely, assume that $\I^R\phi\notin s$, to demonstrate that $\M^c,s\nvDash \I^R\phi$, which by induction hypothesis means that {\em either} $\I\phi\notin s$ {\em or} for all $t\in S^c$, if $sR^{\bullet c}t$ then $\I\phi\in t$. For this, suppose that $\I\phi\in s$, by assumption, we infer that $\neg\I^R\phi\land\I\phi\in s$. Then by the definition of $R^{\bullet c}$, we have that for all $t\in S^c$, if $sR^{\bullet c}t$, then $\I\phi\in t$, as desired.
\end{proof}

With Lindenbaum's Lemma and Truth Lemma in hand, it is then a standard exercise to show the following.
\begin{theorem}[Completeness]\label{thm.k}
$\mathbb{IRIK}$ is sound and strongly complete with respect to the class of all proper bi-frames.
\end{theorem}

\section{Extensions}\label{sec.extensions}

With the minimal logic in hand, it is natural to consider the extensions of $\mathbb{IRIK}$ over five basic class of frames (that is, serial, reflexive, transitive, symmetric, and Euclidean).
In this section, we axiomatize $\IRI$ over the class of serial frames and the class of reflexive frames. Other extensions are left for future work.

\subsection{Serial logic}

We now show that $\mathbb{IRIK}$ is also sound and strongly complete with respect to the class of serial proper bi-frames.

\begin{theorem}
$\mathbb{IRIK}$ is sound and strongly complete with respect to the class of serial proper bi-frames.\footnote{An anonymous referee raised a question whether seriality of bi-frames can be expressed/defined in $\IRI$. As a reply, the answer is negative. This can be shown as in~\cite[Prop.~3.8]{Fanetal:2015}. Despite this, the minimal logic $\mathbb{IRIK}$ of $\IRI$ is sound and strongly complete with respect to the class of serial proper bi-frames. Note that this is not surprising. Another example is contingency logic, in which seriality is not expressed/definable, but the minimal logic $\mathbb{CL}$ is also sound and strongly complete with respect to the class of serial frames, see e.g.~\cite[Thm.~5.6]{Fanetal:2015}.}
\end{theorem}

\begin{proof}
Define $\M^c$ as in Def.~\ref{def.cm}. Note that $R^{\bullet c}$ is serial (in fact, reflexive), but $R^c$ is not serial, since some states $u$ in $\M^c$ may have no $R^c$-successors. We call such states {\em endpoints} w.r.t. $R^c$. We fix this issue with a strategy called `reflexivizing the endpoints'~\cite{Humberstone95,Fanetal:2015}, namely, adding an arrow $R^c$ from each endpoint in $\M^c$ to itself. Formally, $R^{\bf D}=R^c\cup\{(u,u)\mid u\text{ is an endpoint w.r.t. }R^c\text{ in }\M^c\}$. Denote the resultant model as $\M^{\bf D}=\lr{S^c,R^{\bf D},R^{\bullet c},V^c}$. Now $\M^{\bf D}$ is serial. Moreover, for all $\phi\in\IRI$, for all $s\in S^c$, we have $\M^c,s\vDash\phi$ iff $\M^{\bf D},s\vDash\phi$. The nontrivial cases are $\I\phi$ and $\I^R\phi$. If $s$ is an endpoint w.r.t. $R^c$, we have $\M^c,s\nvDash\I\phi$ and $\M^{\bf D},s\nvDash\I\phi$, and thus $\M^c,s\nvDash\I^R\phi$ and $\M^{\bf D},s\nvDash \I^R\phi$; otherwise, the claim is clear. Also, one can see that $\M^{\bf D}$ is also proper, since $R^{\bullet c}$ is reflexive and thus $R^{\bf D}\subseteq R^{\bullet c}$.
\end{proof}

\subsection{Reflexive logic}

Adding the following axiom to $\mathbb{IRIK}$, we obtain the proof system $\mathbb{IRIT}$:
\[
\begin{array}{lr}
(\text{I-T}) & \phi\land\I(\phi\vee\psi)\to\I\phi\\
\end{array}
\]

We will show that $\mathbb{IRIT}$ is sound and strongly complete with respect to the class of reflexive proper bi-frames. For soundness, by Thm.~\ref{thm.k}, it suffices to show the validity of axiom (I-T) on such a class. As a matter of fact, (I-T) is valid on a larger class of frames, namely, the class of reflexive bi-frames (that means we do not restrict the frames to be proper).
\begin{proposition}
Axiom (I-T) is valid on the class of reflexive bi-frames.
\end{proposition}

\begin{proof}
Let $\M=\lr{S,R,R^{\bullet},V}$ be reflexive and $s\in S$. Suppose that $\M,s\vDash\phi\land\I(\phi\vee\psi)$, then $s\vDash\phi$ and there exists $t$ such that $sRt$ and $t\nvDash\phi\vee\psi$, then $t\nvDash\phi$. Since $R$ is reflexive, $sRs$. Therefore, $s\vDash\I\phi$.
\end{proof}

For completeness, again, define $\M^c$ as in Def.~\ref{def.cm}. Although $R^{\bullet c}$ is reflexive, $R^c$ is not, thus $\M^c$ is not reflexive. In this case, the strategy of `reflexivizing the endpoints' does not work, since it cannot guarantee {\em each} world to be reflexive. We need to take the reflexive closure of $R^c$.
\begin{definition}
The {\em canonical model} $\M^{\bf T}=\lr{S^c,R^{\bf T},R^{\bullet c},V^c}$ of $\mathbb{IRIT}$ is the same as $\M^c$ in Def.~\ref{def.cm}, except that $S^c$ consists of all maximal consistent sets of $\mathbb{IRIT}$, and that $R^{\bf T}$ is the reflexive closure of $R^c$.
\end{definition}

It is easy to check that $\M^{\bf T}$ is reflexive and proper.
\begin{lemma}
For all $s\in S^c$, for all $\phi\in\IRI$, we have
$$\M^{\bf T},s\vDash\phi\iff \phi\in s.$$
\end{lemma}

\begin{proof}
By $<^S_d$-induction on $\phi$, where the only nontrivial cases are $\I\phi$ and $\I^R\phi$. The case $\I^R\phi$ can be shown as in Lemma~\ref{lem.truthlemma-k}. For the case $\I\phi$, the ``$\Longleftarrow$'' is similar to the proof of the corresponding part in Lemma~\ref{lem.truthlemma-k} (Note that $R^c\subseteq R^{\bf T}$), and the proof of the ``$\Longrightarrow$'' is as follows.

Suppose that $\M^{\bf T},s\vDash\I\phi$, to show that $\I\phi\in s$. By supposition and induction hypothesis, there are $t,u\in S^c$ such that $sR^{\bf T}t$ and $sR^{\bf T}u$ and $\phi\in t$ and $\phi\notin u$. We consider two cases.
\begin{itemize}
\item $s\neq t$ and $s\neq u$. Then $sR^ct$ and $sR^cu$, and the proof goes as in the corresponding part in Lemma~\ref{lem.truthlemma-k}. We can finally obtain that $\I\phi\in s$.
\item $s=t$ or $s=u$. W.l.o.g. we may assume that $s=t$ (thus $\phi\in s$). Then $s\neq u$, and thus $sR^cu$. From this and $\phi\notin u$, it follows that $\neg\I(\phi\vee\psi)\notin s$ for some $\psi$, and then $\I(\phi\vee\psi)\in s$. Now using axiom (I-T), we derive that $\I\phi\in s$, as desired.
\end{itemize}
\end{proof}

It is now a standard exercise to show the following.
\begin{theorem}\label{thm.comp-t}
$\mathbb{IRIT}$ is sound and strongly complete with respect to the class of reflexive proper bi-frames.
\end{theorem}

\section{Applications}\label{sec.applications}

We now apply our framework to the results in~\cite{Fine:2018}. It turns out that most of the validities mentioned in the introduction are retained here. 

First, as remarked after Def.~\ref{def.minimallogic}, by letting $\chi$ in axiom MIX be $\bot$, we obtain the following result: first-order ignorance plus second-order ignorance implies Rumsfeld ignorance. In other words, if one is both (first-order) ignorant whether and second-order ignorant whether a proposition holds, then one is also Rumsfeld ignorant of that proposition. This indicates an interesting relationship among first-order ignorance, second-order ignorance and Rumsfeld ignorance.
\begin{proposition}\label{prop.provableink}
$\I\phi\land\I\I\phi\to\I^R\phi$ is provable in $\mathbb{IRIK}$.
\end{proposition}

\weg{\begin{proof}
Straightforward by letting $\chi$ in axiom $\text{MIX}$ be $\bot$. 
\end{proof}}

Recall that Fine argues on pages~4032--4033 of~\cite{Fine:2018}, in the setting of the modal system ${\bf S4}$ (which corresponds to $\mathcal{S}4$-frame classes semantically), second-order ignorance implies Rumsfeld ignorance, and the converse also holds (also see (2) and (2$^\ast$) in the Introduction). However, in our framework, only (2) holds (actually, in a context weaker than ${\bf S4}$). We can see the difference between Fine's framework and ours. 

The following result says that in the presence of the axiom $\text{wI-4}$, we can leave out the `(first-order) ignorance' condition in Prop.~\ref{prop.provableink}. That is to say, in the system $\mathbb{IRIK}+\text{wI-4}$, second-order ignorance implies Rumsfeld ignorance. The proof is quite easy.
\begin{proposition}\label{prop.soitori}
$\I\I\phi\to \I^R\phi$ is provable in $\mathbb{IRIK}+\text{wI-4}$, where $\text{wI-4}$ is $\I\I\phi\to\I\phi$.
\end{proposition}

\weg{\begin{proof}
We have the following proof sequence in $\mathbb{IRIK}+\text{wI-4}$.
\[
\begin{array}{lll}
    (i) & \I\phi\land\I\I\phi\to\I^R\phi & \text{Prop.~}\ref{prop.provableink}\\
    (ii) & \I\I\phi\to\I\phi & \text{wI-4}\\
    (iii) & \I\I\phi\to\I^R\phi & (i),(ii)\\
\end{array}
\]
\end{proof}}

As Fine shows, in the setting of ${\bf S4}$, Rumsfeld ignorance implies second-order ignorance. However, under our restriction that $R\subseteq R^{\bullet}$, this does not hold.

\begin{proposition}\label{prop.ritosoi}
$\I^R\phi\to\I\I\phi$ is not valid over ${\bf S4}$ proper bi-frames, thus neither is $\I^R\phi\lra \I\I\phi$.
\end{proposition}

\begin{proof}
\weg{Consider the following bi-model $\M=\lr{S,R,R^{\bullet},V}$, where $R\subseteq R^{\bullet}$. For simplicity, we omit the reference to $R^\bullet$, which holds on all arrows.
$$
\xymatrix{&t:p\ar@(ur,ul)_R\\
s:p\ar[ur]\ar[r]_R\ar[dr]_R\ar@(ul,ur)^{R} & u:p\ar[d]^R\ar@(ur,ul)^R\\
&v:\neg p\ar[u]\ar@(dr,dl)^R\\}
$$}
Take the reflexive closure of the bi-model $\M$ in the proof of Prop.~\ref{prop.bimodel-validities}(i), denoted $\M^{\bf T}$.
One may check that  $\M^{\bf T}$ is indeed ${\bf S4}$ (that is, both $R$ and $R^{\bullet}$ are reflexive and transitive) and proper. Moreover, one may verify that $\M^{\bf T},s\vDash\I^Rp$ and $\M^{\bf T},s\nvDash\I\I p$. This shows that $\I^Rp\to \I\I p$ is not valid over ${\bf S4}$ proper bi-frames, as desired.
\end{proof}

\weg{\begin{proof}
We have the following proof sequence in $\mathbb{IRIT}$.
\[
\begin{array}{lll}
    (i) & \I_a\phi\to\I_a\phi & \text{TAUT} \\
    (ii) & \neg\I_a\I_a\phi\land \neg\I_a(\I_a\phi\vee\phi)\to\neg\I_a^R\phi &(i),\text{R-MIX}\\
    (iii) & \I^R_a\phi\to \I_a\I_a\phi\vee\I_a(\I_a\phi\vee\phi) & (ii)\\
    (iv) & \I_a(\I_a\phi\lor \phi)\land \I_a\phi\to\I_a\I_a\phi\vee\I_a(\I_a\phi\to\I_a\phi\vee\phi) &\text{I-T}\\
    (v) & \I_a\phi\to\I_a\phi\vee\phi &\text{TAUT}\\
    (vi) & \neg I_a(\I_a\phi\to\I_a\phi\vee\phi)& (v),\text{R-NI}\\
    (vii) & \I_a(\I_a\phi\lor \phi)\land \I_a\phi\to\I_a\I_a\phi & (iv),(vi)\\
    (viii) & \I_a\phi\to \I_a\I_a\phi\vee\neg\I_a(\I_a\phi\lor \phi)&(vii)\\
    (ix) & \I^R_a\phi\to\I_a\phi & \text{RI-I}\\
    (x) & \I^R_a\phi\to\I_a\I_a\phi & (iii),(viii),(ix)\\
\end{array}
\]
\end{proof}}


Fine~\cite[p.~4033]{Fine:2018} argues that within the context of the system ${\bf S4}$, one does not know that one is Rumsfeld ignorant (see also (3) in the Introduction). This can be symbolized as $\neg\Box\I^R\phi$.\footnote{$\Box$ represents ``knowing that'', which is interpreted by the following: given a pointed bi-model $(\M,s)$, $\M,s\vDash\Box\phi$ iff for all $t$, if $sR^\bullet t$ then $\M,t\vDash\phi$.} This is also the case in our framework. Recall that $R^\bullet$ is the accessibility relation for `knowing that'.
\begin{proposition}
$\neg\Box\I^R\phi$ is valid over ${\bf S4}$ proper bi-frames.
\end{proposition}

\begin{proof}
Suppose not, that is, there exists ${\bf S4}$ proper bi-model $\M=\lr{S,R,R^\bullet,V}$ and $s\in S$ such that $\M,s\vDash \Box\I^R\phi$. Then by reflexivity of $R^{\bullet}$, $\M,s\vDash\I^R\phi$. From this, it following that there exists $t$ such that $sR^\bullet t$ and $\M,t\nvDash\I\phi$. However, since $s\vDash\Box\I^R\phi$ and $sR^\bullet t$, we have $t\vDash\I^R\phi$, which implies $t\vDash\I\phi$: a contradiction.
\end{proof}

Note that in the above proof, we only use the reflexivity of $R^\bullet$, so this actually shows that $\neg\Box\I^R\phi$ is valid over reflexive proper bi-frames. However, we cannot show it in a syntactic way; that is to say, we cannot show that $\neg\Box\I^R\phi$ is provable in the system $\mathbb{IRIT}$, since the operator $\Box$ is undefinable in $\IRI$ over the class of ${\bf S4}$ proper bi-frames.\footnote{For instance, take two ${\bf S4}$ proper bi-models $\M=\lr{S,R,R^\bullet, V}$ and $\M'=\lr{S',R',R^{\bullet'}, V'}$, where $S=\{s,t\}$, $R(s)=\{s\}$, $R^\bullet(s)=\{s,t\}$, $R(t)=R^\bullet(t)=\{t\}$, $V(p)=\{s,t\}$, whereas $S'=\{s',t'\}$, $R'(s')=\{s'\}$, $R^{\bullet'}(s')=\{s',t'\}$, $R'(t')=R^{\bullet'}(t')=\{t'\}$,  $V(p)=\{s'\}$. One may show that $(\M,s)$ and $(\M',s')$ cannot be distinguished by $\IRI$-formulas, but can be distinguished by $\Box p$.}

Also, Fine~\cite[Thm.~4]{Fine:2018} proves that within the context of ${\bf S4}$, one does not know one is second-order ignorant. This may be symbolized as $\neg\Box\I\I\phi$. Again, this is also the case in our framework.
\begin{proposition}
$\neg\Box\I\I\phi$ is valid over ${\bf S4}$ proper bi-frames.
\end{proposition}

\begin{proof}
Suppose not, then there exists ${\bf S4}$ proper bi-model $\M=\lr{S,R,R^\bullet,V}$ and $s\in S$ such that $\M,s\vDash \Box\I\I\phi$. Since $sR^\bullet s$, $\M,s\vDash \I\I\phi$. It follows that for some $t,u$ such that $sRt$ and $sRu$, we have $\M,t\vDash\I\phi$ and $\M,u\nvDash\I\phi$. Since $R\subseteq R^\bullet$, $sR^\bullet u$, and then $\M,u\vDash\I\I\phi$. So there is a $v$ such that $uRv$ and $\M,v\vDash \I\phi$, and thus $vRv_1$ and $vRv_2$ and $\M,v_1\vDash\phi$ and $\M,v_2\nvDash\phi$ for some $v_1$ and $v_2$. By transitivity of $R$, we obtain $uRv_1$ and $uRv_2$, and thus $\M,u\vDash\I\phi$: a contradiction.
\end{proof}

\weg{\begin{fact}\label{prop.iitoiii}
$\I\I\phi\to\I\I\I\phi$ is valid over ${\bf S4}$ proper bi-frames.
\end{fact}}

Moreover, Fine~\cite[Coro.~5]{Fine:2018} shows that all $n$-th (where $n\geq 3$ and $n\in \mathbb{N}$) order ignorance are reduced to second-order ignorance. This is also the case in our framework.

\begin{proposition}\label{coro.reducedtosecond}
For natural numbers $m,n>1$, $\I^m\phi\lra \I^n\phi$ is valid over ${\bf S4}$ proper bi-frames.
\end{proposition}

\begin{proof}
Let $m,n>1$ be natural numbers.
It is not hard to show that $\I\I\phi\to\I\I\I\phi$ is valid over ${\bf S4}$ proper bi-frames (in this case, we only need the ${\bf S4}$ property of $R$). In consequence, $\I^m\phi\to\I^n\phi$ is provable in $\mathbb{IRIS}4$ for $m\leq n$. Furthermore, we can show that $\I\I\phi\to\I\phi$ is valid over ${\bf S4}$ proper bi-frames (in this case, we only need the property of transitivity of $R$), so is
$\I^m\phi\to\I^n\phi$ for $m>n$, as desired.
\end{proof}

We conclude this section with another application of our framework. Fine~\cite[Thm.~7]{Fine:2018} demonstrates that within the context of ${\bf S4}$, one does not know that any of the following formulas hold: $\I(\I\phi_1\land\cdots\land\I\phi_n)$, $\I(\I\phi_1\lor\cdots\lor\I\phi_n)$, $\I\I\phi_1\land\cdots\land\I\I\phi_n$, and $\I\I\phi_1\lor\cdots\lor\I\I\phi_n$. Again, this is also the case in our framework.
\begin{proposition}
Let $n>0$ be a natural number. Then the following formulas are valid over ${\bf S4}$ proper bi-frames.
\begin{itemize}
    \item[(1)] $\neg\Box\I(\I\phi_1\land\cdots\land\I\phi_n)$
    \item[(2)] $\neg\Box\I(\I\phi_1\lor\cdots\lor\I\phi_n)$
    \item[(3)] $\neg\Box(\I\I\phi_1\land\cdots\land\I\I\phi_n)$
   \item[(4)] $\neg\Box(\I\I\phi_1\lor\cdots\lor\I\I\phi_n)$
\end{itemize}
\end{proposition}

\begin{proof}
Let $\M=\lr{S,R,R^\bullet,V}$ be an ${\bf S4}$ proper bi-model and $s\in S$ be arbitrary.

For (1), suppose, for a contradiction, that $\M,s\vDash\Box\I(\I\phi_1\land\cdots\land\I\phi_n)$. By reflexivity of $R^\bullet$, $s\vDash\I(\I\phi_1\land\cdots\land\I\phi_n)$. Then there exists $t$ such that $sRt$ and $t\nvDash \I\phi_1\land\cdots\land\I\phi_n$, and thus $t\nvDash\I\phi_m$ for some $m=1,\dots,n$. Since $R\subseteq R^\bullet$, we infer that $sR^\bullet t$, and then $t\vDash\I(\I\phi_1\land\cdots\land\I\phi_n)$. Then for some $u$ such that $tRu$, we have $u\vDash\I\phi_1\land\cdots\land\I\phi_n$, thus $u\vDash\I\phi_m$. This follows that $uRu_1$ and $uRu_2$ and $u_1\vDash\phi_m$ and $u_2\nvDash\phi_m$ for some $u_1$ and $u_2$. By transitivity of $R$, $tRu_1$ and $tRu_2$, and hence $t\vDash\I\phi_m$: a contradiction. Therefore, the formula (1) is valid over ${\bf S4}$ proper bi-frames.
Similarly, we can show the validity of (2) over the frames in question.

For (3), assume towards contradiction that $\M,s\vDash\Box(\I\I\phi_1\land\cdots\land\I\I\phi_n)$. By reflexivity of $R^\bullet$, $s\vDash\I\I\phi_1\land\cdots\land\I\I\phi_n$, thus $s\vDash\I\I\phi_1$. Then there are $t$ such that $sRt$ and $t\nvDash\I\phi_1$. Since $R\subseteq R^\bullet$, we have $sR^\bullet t$, thus $t\vDash\I\I\phi_1$. But in the proof of Prop.~\ref{coro.reducedtosecond} we have seen that $\I\I\phi_1\to \I\phi_1$ is valid over $\M$: a contradiction.

The proof for (4) is more complicated. Suppose, for a contradiction, that $\M,s\vDash\Box(\I\I\phi_1\vee\cdots\vee\I\I\phi_n)$. By reflexivity of $R^\bullet$, $s\vDash\I\I\phi_1\vee\cdots\vee\I\I\phi_n$. Then $\M,s\vDash\I\I\phi_i$ for some $i=1,\dots,n$. W.l.o.g. we may assume that $i=1$, that is, $\M,s\vDash\I\I\phi_1$. Then there exists $s_1$ such that $sRs_1$ and $\M,s_1\nvDash\I\phi_1$, and thus $s_1\nvDash\I\I\phi_1$. Since $R\subseteq R^\bullet$, $sR^\bullet s_1$, thus $s_1\vDash\I\I\phi_1\vee\cdots\vee\I\I\phi_n$, and hence $s_1\vDash\I\I\phi_2\vee\cdots\vee\I\I\phi_n$. Again, w.l.o.g. we may assume that $s_1\vDash\I\I\phi_2$. Then there exists $s_2$ such that $s_1Rs_2$ and $\M,s_2\nvDash\I\phi_2$, and thus $s_2\nvDash\I\I\phi_2$. Again, since $R\subseteq R^\bullet$, $s_1R^\bullet s_2$. By transitivity of $R^\bullet$, $sR^\bullet s_2$, thus $s_2\vDash\I\I\phi_1\vee\cdots\vee\I\I\phi_n$. Moreover, by reflexivity of $R$, $s_1Rs_1$, from which we can obtain that $s_2\nvDash\I\phi_1$, and then $s_2\nvDash\I\I\phi_1$. This follows that $s_2\vDash\I\I\phi_3\vee\cdots\vee\I\I\phi_n$. Continuing with this process, we can obtain a sequence $s_2Rs_3R\cdots Rs_{n}$ in $\M$ and $s_{n}\nvDash\I\I\phi_1\vee\cdots\vee\I\I\phi_{n}$. However, by $R\subseteq R^\bullet$ and transitivity of $R^\bullet$, we infer that $sR^\bullet s_n$, and thus $s_n\vDash\I\I\phi_1\vee\cdots\vee\I\I\phi_n$: a contradiction. 
\end{proof}

\section{Concluding words}\label{sec.conclusion}

In this paper, we demonstrated that under the usual semantics, Rumsfeld ignorance is definable in terms of ignorance, which makes some existing results and the axiomatization problem trivial. A main reason is that the accessibility relations for the implicit knowledge operator contained in the packed operators of ignorance and Rumsfeld ignorance are the same. We then proposed a bi-semantics, in which the two accessibility relations are different. We showed that under the bi-semantics with the restriction that one relation is an arbitrary subset of the other, Rumsfeld ignorance is undefinable with ignorance, and we can also retain most of the validities in the case when the two accessibility relations are the same. We then axiomatized the logic of ignorance and Rumsfeld ignorance over various proper bi-frame classes. We finally applied our framework to analyze Fine's results in~\cite{Fine:2018}.

Although in our framework, some Fine's validities, for instance $(2^\ast)$ (mentioned in the introduction),  are lost, most of Fine's validities, for instance $(1)$, $(2)$, $(3)$, and $(4)$, are retained; most importantly, we avoid the issues of definability of Rumsfeld ignorance with ignorance and the triviality of the axiomatization problem of the logic of ignorance and Rumsfeld ignorance.

Note that our axiomatizations are infinitary. For the future work, we hope to find a finite axiomatization for the minimal logic of ignorance and Rumsfeld ignorance, and its extensions over other special frame classes, such as S4, KD45 and S5. Another work is axiomatizing the logic of propositional logic with Rumsfeld ignorance as a sole primitive modality. Moreover, when it comes to the multi-agent case, $\neg\I_ap\land\I_b\neg\I_ap$ is read that ``$a$ knows whether $p$ but $b$ is ignorant about that'', which seemingly expresses a kind of secret. We can investigate the logical properties of such kind of secret.

\subsection*{Acknowledgements}

This article is supported by the Fundamental Research Funds for the Central Universities. The author acknowledges two anonymous referees of this journal for their insightful and helpful comments. Special thanks go to Lloyd Humberstone and Yanjing Wang for their very helpful comments. The author also thanks three anonymous referees of NCML 2022 and the audience of that conference, where an earlier version of this article is presented.

\bibliographystyle{plain}
\bibliography{biblio2019,biblio2016}

\end{document}